\documentclass[12pt,reqno]{amsproc}
\usepackage{amsthm,wrapfig,amsmath,amssymb,amsfonts,setspace,verbatim,graphicx,mathtools,mathrsfs,commath,float,mdframed,frame,xcolor,qtree,enumitem, ulem, afterpage}
\usepackage[hidelinks]{hyperref}
\usepackage[T1]{fontenc}

\DeclareMathOperator{\dom}{dom}

\DeclareMathOperator{\addr}{addr}

\DeclareMathOperator{\eal}{Re}
\DeclareMathOperator{\imag}{Im}
\DeclareMathOperator{\suc}{succ}
\usepackage[a4paper,margin=1.1in]{geometry}
\usepackage{lmodern}

\newtheorem{ut}{Theorem}

\numberwithin{ut}{section}
\numberwithin{equation}{section}
\newtheorem{up}[ut]{Proposition}

\newtheorem{uc}[ut]{Corollary}
\newtheorem{ucl}[ut]{Claim}

\theoremstyle{definition}

\newtheorem{ue}[ut]{Example}

\begin{document}

\title{Erd\H{o}s space in Julia sets}

\address{Department of Mathematics, Auburn University at Montgomery, Montgomery 
AL 36117, United States of America}
\email{dsl0003@auburn.edu; dlipham@aum.edu}
\author{David S. Lipham}

\begin{abstract}We prove that the rational Hilbert space $\mathfrak E$, known as \textit{Erd\H{o}s space},   surfaces in complex dynamics via iteration of  $e^z-1$.  \end{abstract}

\maketitle

\section{Introduction}

The subspace of  $\ell^2$  consisting of all rational sequences, $$\mathfrak E:=\{\mathbf x\in \ell^2:x_n\in \mathbb Q\text{ for all }n<\omega\},$$ was introduced by Paul Erd\H{o}s  in order to show that squaring a topological space of positive dimension does not necessarily raise the dimension. In fact,  \textit{Erd\H{o}s space} $\mathfrak E$  is $1$-dimensional \cite{dims} and homeomorphic to all of its powers, up to and including $\mathfrak E^\omega$ \cite{erd}.  Dijkstra and van Mill \cite{erd} also found representations of $\mathfrak E$ in other contexts. For example, they proved that $\mathfrak E$ is topologically equivalent  to the space of homeomorphisms of $\mathbb R ^2$ which map $\mathbb Q ^2$ onto itself.  Our goal in this paper  is to show  that $\mathfrak E$  surfaces in complex dynamics. Moreover we will show that $\mathfrak E$ is generated by the  transcendental entire function $f(z)=e^z-1.$


The Julia set\footnote{The \textit{Julia set} of an entire function $f$ is traditionally defined to be the set of non-normality for the family $\{f^n:n\in \mathbb N\}$. For $f(z)=e^z-1$ it is equal to the set $J(f)$ defined above, owing to the fact that all points in the left half-plane attract to $0$.} of $f$ consists of all complex numbers whose orbits stay in the right half-plane; $$J(f)=\bigcap _{n=1}^\infty \{z\in \mathbb C:\eal(f^n(z))\geq 0\}.$$
 The first three sets from this intersection are depicted in Figure 1. Each consists of an infinite  collection of domains,   which are cut into smaller domains by incrementing $n$. Every descending sequence of domains accumulates onto  a simple curve  that connects  a point of $\mathbb C$ to the point at infinity.  Informally, this shows that $J(f)$ is a collection of mutually separated curves with endpoints.  
  See Devaney and Krych \cite{dev,bif} or Aarts and Oversteegen \cite{aa} for more about the structure of $J(f)$.

Let $E(f)$ be the set of all finite endpoints of maximal curves in $J(f)$.  
 Kawamura, Oversteegen and Tymchatyn   \cite{31} proved that $E(f)$ is homeomorphic to \textit{complete Erd\H{o}s space} $\mathfrak E_{\mathrm{c}}:=\{\mathbf x\in \ell^2:x_n\in \mathbb R\setminus \mathbb Q\text{ for all }n<\omega\}$. The main result of this paper is:
\begin{ut} The imaginary-escaping endpoint set $$\ddot E(f)=\{z\in E(f):\imag(f^n(z))\to\infty\}$$ is homeomorphic to $\mathfrak E$.\end{ut}\noindent This  provides a partial positive answer to \cite[Question 1]{lip} insofar as it shows that the escaping endpoint set 
 $\dot E(f)=\{z\in E(f):|f^n(z)|\to\infty\}$ contains a dense copy of $\mathfrak E$ (although $\dot E(f)$ is homeomorphic to neither $\mathfrak E$ nor $\mathfrak E_{\mathrm{c}}$ \cite{lipp}). 
\begin{figure}[h]
\centering
\includegraphics[scale=0.55]{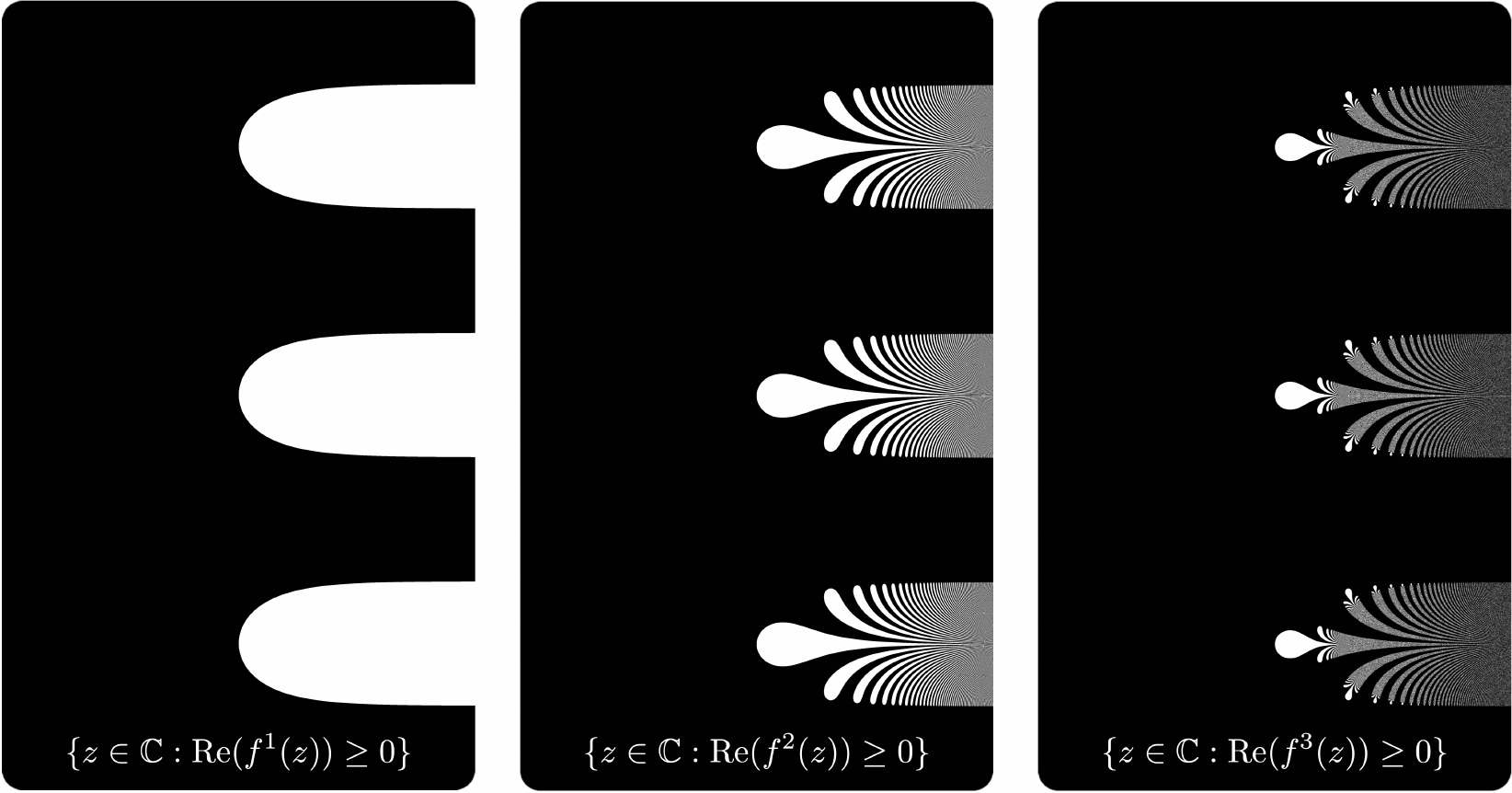}
\caption{Approximations of $J(f)$ in the window where $|\eal(z)|\leq 6$ and $|\imag(z)|\leq 10$. }
\end{figure}

\subsection*{Sketch of proof}The   proof of Theorem 1.1 will be largely based on Rempe \cite{rem2}, Dijkstra and van Mill \cite{erd}, and  Alhabib and Rempe \cite{rem}.  The topological model of $J(f)$ in \cite{rem2}  will be used to define a function $\psi$ whose graph $G^\psi_\infty$  is homeomorphic to $\ddot E(f)$.  Then we will apply the extrinsic characterization of $\mathfrak E$ from \cite{erd} to the space $G^\psi_\infty$. In order to show that the characterization fully applies to the model,  we will require several results from \cite{rem}. The most significant is   \cite[Theorem 3.6]{rem} which is essentially a generalization of  the fact that $E(f)$ is dense in $J(f)$.

All  relevant results from \cite{erd} and \cite{rem2} are presented in Sections 2 and 3. Those in \cite{rem} will be cited as needed during the proof, which is in Section 4.

\begin{figure}[h]
\centering
\includegraphics[scale=.45]{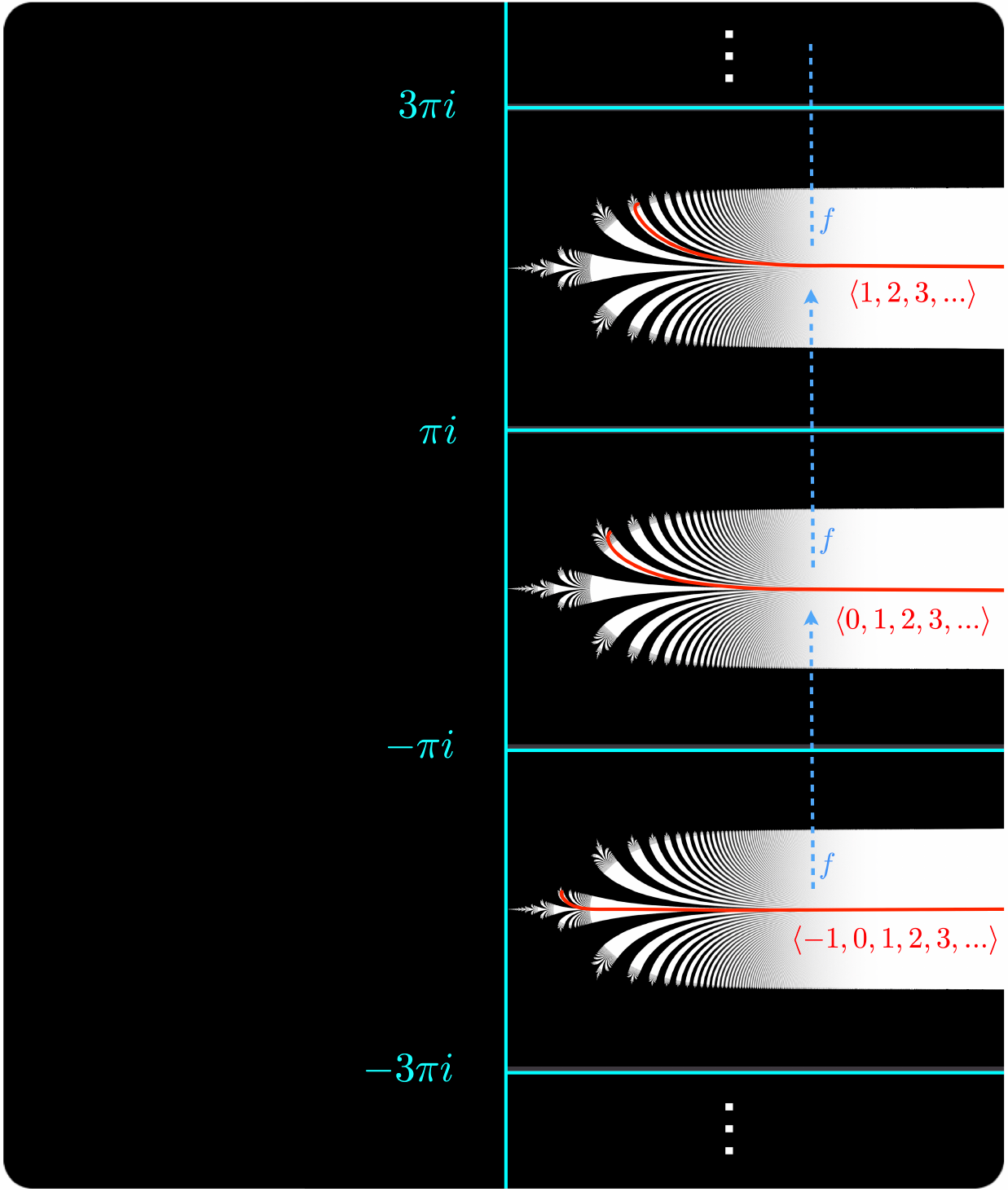}
\caption{$J(f)$ (in white), and the external addresses of three curves (in red) whose iterates tend to $\infty$ in the imaginary direction.}
\end{figure}

\section{Topological models of $J(f)$ and $\ddot E(f)$}

\subsection{External addresses}

Let $\mathbb Z ^\omega$ be the space of  integer sequences $\uline{s}=\langle s_0,s_1,s_2,\ldots\rangle$.  We say that a point $z\in \mathbb C$  has \textit{external address} $\uline s$ (in symbols, $\addr(z)=\uline s$) if $$\imag(f^n(z))\in [(2s_n-1)\pi,(2s_n+1)\pi]$$ for each $n<\omega$. 

If   $\uline s\in \mathbb Z ^\omega$ and $|s_n|$ does not increase at a faster-than-exponential rate, then $\uline s$  is the external address of a unique curve of $J(f)$; see \cite[Section 3]{bif} or \cite[Proposition 3.2]{sz}.   

\begin{ue}The sequence $\langle -1, 0,1,2,3,\ldots\rangle$ increases at only a linear rate, and is therefore the external address of a curve  in $J(f)$.  Figure 2 shows the ray at that address, followed by two of its iterates. Note that the  endpoint at this address belongs to $\ddot E(f)$.  In particular, this shows that $\ddot E(f)$ is non-empty.
\end{ue}


\subsection{Model of the Julia set}Define
$\mathcal F:[0,\infty)\times \mathbb Z ^\omega\to \mathbb R\times \mathbb Z ^\omega$ by $$\langle t,\uline{s}\rangle\mapsto \langle F(t)-2\pi|s_1|,\sigma(\uline{s})\rangle,$$ where $F(t)=e^t-1$ and   $\sigma$ is the shift  on $\mathbb Z ^\omega$ (i.e.\ $\sigma(\langle s_0,s_1,s_2,\ldots\rangle)=\langle s_1,s_2,s_3,\ldots\rangle$). Put $T(x)=t$ for each $x=\langle t,\uline{s}\rangle\in [0,\infty)\times \mathbb Z ^\omega$, and let  $$J(\mathcal F)=\{x\in [0,\infty)\times \mathbb Z ^\omega:T(\mathcal F^n(x))\geq 0\text{ for all }n\geq 0\}.$$

If  $\uline{s}\in \mathbb Z ^\omega$ and there exists $t\geq 0$ such that $\langle t,\uline{s}\rangle\in J(\mathcal F)$, then let $$t_{\uline{s}}=\min\{t\geq 0:\langle t,\uline{s}\rangle\in J(\mathcal F)\}.$$ Otherwise, put $t_{\uline{s}}=\infty$.  Observe that $$J(\mathcal F)=\bigcup _{\uline s\in \mathbb Z^\omega}[t_{\uline s},\infty)\times \{\uline s\}.$$ Thus the  points $\langle  t_{\uline{s}},\uline{s}\rangle$ with $t_{\uline s}<\infty$  are the (finite) endpoints of $J(\mathcal F)$. 

The following is  implicit in \cite{rem2}.

\begin{up}\label{ppl}There is a homeomorphism $H:J(\mathcal F)\to J(f)$ such that $$\addr(H(\langle t,\uline{s}\rangle))=\uline s$$ for every $\langle t,\uline{s}\rangle\in J(\mathcal F)$.\end{up}


\begin{proof}Let $H$ be the mapping defined in \cite[Theorem 9.1]{rem2} for the parameter $\kappa=-1$.  By  construction, $H$ is one-to-one on a set $X$ which contains all non-endpoints of $J(\mathcal F)$ \cite[Observation 3.1]{rem2}.  Since  $J(f)$ is a union of disjoint copies of $[0,\infty)$, it follows that all of $H$ is one-to-one. In the proof of  \cite[Theorem 9.1]{rem2} it is also noted that $H$ is closed  because it extends to a mapping of the one-point compactifications. So $H$ is a homeomorphism.  

Now let  $\langle t,\uline{s}\rangle\in J(\mathcal F)$.   By the remark after the statement of \cite[Theorem 9.1]{rem2}  and the construction of $\mathfrak g$ in \cite[Section 4]{rem2}, there exists $t'>t$ such that $H(\langle t',\uline{s}\rangle)=\mathfrak g(\langle t',\uline{s}\rangle)$.    By  \cite[Theorem 4.2]{rem2},    $\addr(\mathfrak g(\langle t',\uline{s}\rangle))=\uline{s}$. 
Since $H(\langle t,\uline{s}\rangle)$ and $H(\langle t',\uline{s}\rangle)$ lie on the same curve, and each curve of  $J(f)$ is contained in a horizontal strip of the form $\{z\in \mathbb C:(2k-1)\pi<\imag(z)<(2k+1)\pi\}$,  this implies $\addr(H(\langle t,\uline{s}\rangle))=\uline s$. \end{proof}

\subsection{Model of the imaginary-escaping endpoints}

\begin{uc}$\{\langle t_{\uline s},\uline s\rangle\in J(\mathcal F):s_n\to\infty\}\simeq \ddot E(f)$.\end{uc}

\begin{proof}Let $H$ be the homeomorphism from Proposition \ref{ppl}. For any endpoint $\langle t_{\uline s},\uline s\rangle\in J(\mathcal F)$ it is evident that $H(\langle t_{\uline s},\uline s\rangle)\in E(f)$. Note also that if $z\in J(f)$ and $\uline s=\addr(z)$, then $\imag(f^n(z))\to\infty$ if and only if $s_n\to\infty$. Thus $H(\{\langle t_{\uline s},\uline s\rangle\in J(\mathcal F):s_n\to\infty\})= \ddot E(f)$.\end{proof}

\section{Characterizations of $\mathfrak E$}

\subsection{Sierpi\'{n}ski stratification}The characterizations of $\mathfrak E$ in \cite[Section 7]{erd} involve trees of closed subsets that are called Sierpi\'{n}ski stratifications. They are defined as follows.

For any set $A$ we let $A^{<\omega}$ denote the set of all functions $\alpha$ such that $\dom(\alpha)<\omega$ (the domain of $\alpha$ is a finite ordinal), and the range of $\alpha$ is a subset of $A$. Thus if $\alpha\in A^{<\omega}$ and $n=\dom(\alpha)$ then $\alpha$ is an $n$-tuple of elements of $A$; $\alpha=\langle a(0),\ldots ,a(n-1)\rangle$. We shall write $\alpha\prec \beta$ if $\dom(\alpha)<\dom(\beta)$ and the restriction  $\beta\restriction \dom(\alpha)$ is equal to  $\alpha$ (that is,  $\beta(i)=\alpha(i)$ for all $i<\dom(\alpha)$).

A \textit{tree $T$ on an alphabet }$A$ is a subset of $A^{<\omega}$ that is closed under initial segments, i.e.\ if $\beta \in T$ and $\alpha\prec \beta$ then $\alpha\in T$.  An element $\lambda\in A^\omega$ is an \textit{infinite branch} of $T$ provided $\lambda\restriction k \in T$ for every $k <\omega$. We let  $[T]$ denote  the set of all infinite branches of $T$. If $\alpha,\beta\in T$ are such that $\alpha\prec \beta$ and $\dom(\beta)=\dom(\alpha)+1$, then we say that $\beta$ is an \textit{immediate successor} of $\alpha$  and $\suc(\alpha)$ denotes the set of immediate successors of $\alpha$ in $T$. 

Let $X$ be a non-empty separable metrizable space. A system $(X_\alpha)_{\alpha\in T}$ is called a \textit{Sierpi\'{n}ski stratification} of  $X$ if:
\begin{enumerate}
 \item $T$ is a tree over a countable alphabet,

 \item each $X_\alpha$ is a closed subset of $X$,   
 
\item  $X_\varnothing = X$  and $X_\alpha=\bigcup \{X_\beta:\beta\in \suc(\alpha)\}$ for each $\alpha\in T$,  and  
 
 \item  if $\lambda \in [T]$ then the sequence $X_{\lambda\restriction 0}, X_{\lambda\restriction 1}, \ldots$  converges to a point  in $X$.
\end{enumerate}

\subsection{Graphs of upper-semicontinuous functions homeomorphic to $\mathfrak E$} A function $\varphi : X \to[0,\infty]$ is \textit{upper semi-continuous}   if $\varphi^{-1}[0,t)$ is open in $X$ for every $t\in \mathbb R$. 

An upper semi-continuous  function $\varphi:X\to [0,1]$ is called a \textit{Lelek function} if $X$ is zero-dimensional,  $X' = \{x \in X:\varphi(x)>0\}$ is dense in $X$, and   $$G^\varphi_0:=\{\langle x,\varphi(x)\rangle:\varphi(x)>0\}$$ is dense in $$
L^{\varphi\restriction X'}_0:=\bigcup_{x\in X'} \{x\}\times[0,\varphi(x)].$$
A Lelek function $\varphi: X \to [0,1]$   belongs to the \textit{Sierpi\'{n}ski-Lelek class} \textsf{SL} if there exists a Sierpi\'{n}ski stratification $(X_\alpha)_{\alpha\in T}$ of $X$ such that:  \begin{enumerate}
\item[(a)] $\varphi\restriction X_\alpha$ is Lelek for every $\alpha\in T$; and 
 \item[(b)] $G^{\varphi\restriction X_\beta}_0$ is nowhere dense in $G^{\varphi\restriction X_{\alpha}}_0$ for each $\alpha\in T$ and $\beta\in \suc(\alpha)$. 
 \end{enumerate}

\begin{up}[{\cite[Theorem 7.12]{erd}}]\label{t3}If $\varphi\in \textnormal{\textsf{SL}}$, then $G^\varphi_0\simeq \mathfrak E$. \end{up}

\begin{ue}Define $\eta:\mathbb Q ^\omega\to [0,1)$ by $\eta(\mathbf x)=\frac{1}{1+\|\mathbf x\|}$, where $$\|\mathbf x\|=\sqrt{\sum_{n=0}^\infty x_n^2}$$ is the $\ell^2$-norm of $\mathbf x$, and $1/\infty=0$.  Let $T=\mathbb Q^{<\omega}$. For each $\alpha=\langle q_0,\ldots ,q_{n-1}\rangle\in T$ define $$X_\alpha=\{q_0\}\times \ldots \times \{q_{n-1}\}\times \mathbb Q\times \mathbb Q\times \ldots.$$ The system $(X_\alpha)_{\alpha\in T}$ witnesses that $\eta\in \textsf{SL}$ (see \cite[Proposition 7.11]{erd}). Hence $G^\eta_0\simeq \mathfrak E$. 
\end{ue}

\subsection{Graphs of lower-semicontinuous functions homeomorphic to $\mathfrak E$}In order to prove Theorem 1.1, we will require a lower semi-continuous version of Proposition \ref{t3}. 

A function $\psi : X \to[0,\infty]$ is  \textit{lower semi-continuous} if $\psi^{-1}(t,\infty]$ is open in $X$ for every $t \in\mathbb R$. Define  \begin{align*}G^\psi_\infty&=\{\langle \psi(x),x\rangle:\psi(x)<\infty\} \text{; and} \\
L^\psi_\infty&=\bigcup _{x\in X}[\psi(x),\infty]\times \{x\}
.\end{align*}
Note that the first coordinate is now the output of the function. This is done so that $L^\psi_\infty$ is a collection of horizontal arcs resembling $J(\mathcal F)$. 

\begin{up}\label{t4}Let $\psi:X\to [0,\infty]$ be a lower semi-continuous function with zero-dimensional domain $X$.   Suppose    $(X_\alpha)_{\alpha\in T}$ is a Sierpi\'{n}ski stratification of $X$,    \begin{enumerate} 
\item[\textnormal{(a)}] $X_\alpha ':=\{x\in X_\alpha:\psi(x)<\infty\}$ is dense in $X_\alpha$ for each $\alpha\in T$;
\item[\textnormal{(b)}]$G^{\psi\restriction X_\alpha}_\infty$ is dense in $L^{\psi\restriction X'_\alpha}_\infty$ for each $\alpha\in T$; and  
\item[\textnormal{(c)}]  $G^{\psi\restriction X_{\beta}}_\infty$ is nowhere dense in $G^{\psi\restriction X_{\alpha}}_\infty$  for each $\alpha\in T$ and $\beta\in \suc(\alpha)$.
\end{enumerate} 
Then  $G^\psi_\infty\simeq \mathfrak E$.
\end{up}

\begin{proof}Observe that $\varphi:=1/(1+\psi)\in \textsf{SL}$ and $G^\psi_\infty\simeq G^{\varphi}_0$.  So the conclusion $G^\psi_\infty\simeq \mathfrak E$ follows from Proposition \ref{t3}. \end{proof}

\begin{ue}The conditions of Proposition \ref{t4} are all satisfied if  $(X_{\alpha})_{\alpha\in T}$ is the system from Example 3.2,  and  $\psi:\mathbb Q ^\omega\to [0,\infty]$ is defined by $\psi(\mathbf x)=\|\mathbf x\|$.
\end{ue}

\section{Proof of Theorem 1.1}

Let $X=\{\uline{s}\in \mathbb Z ^\omega:s_n\to\infty\}.$ Define $\psi:X\to [0,\infty]$ by $\psi(\uline{s})=t_{\uline{s}}$.   The function $\psi$ is lower semi-continuous  \cite[Observation 3.1]{rem2}, and  $\ddot E(f)\simeq G^\psi_\infty$ by Corollary 2.3.  It remains to show $G^\psi_\infty\simeq \mathfrak E$.   This will be accomplished by constructing  a Sierpi\'{n}ski stratification of  $X$ so that the conditions in Proposition \ref{t4} are satisfied.  


Our tree $T$ will be a subset of  $(\mathbb N\times \mathbb Z)^{<\omega}$.   We may identify each    $\alpha\in (\mathbb N\times \mathbb Z)^{<\omega}$ with an  $n$-tuple of ordered pairs $\langle \langle N_0,s_0\rangle, \langle N_1,s_1\rangle,\ldots,\langle N_{n-1},s_{n-1}\rangle \rangle$ where $n=\dom(\alpha)$ and $\langle N_i,s_i\rangle=\alpha(i)$. Given $\alpha\in (\mathbb N\times \mathbb Z)^{<\omega}$ we define $N(\alpha(i))=N_i$ and  $$\alpha^\frown \langle N,s\rangle=\langle \alpha(0),\alpha(1),\ldots,\alpha(n-1),\langle N,s\rangle\rangle.$$

 Let $X_\varnothing =X$.  Supposing  $X_\alpha$ has been defined,  for each $\langle N,s\rangle\in \mathbb N\times \mathbb Z$ let  $$X_{\alpha^\frown \langle N,s\rangle}=\{{\uline{s}}\in X_\alpha:s_{\dom(\alpha)}=s\text{ and }s_n\geq \dom(\alpha)+1\text{ for all }n\geq N\}.$$ In this manner, $X_{\alpha}$ is recursively defined for every $\alpha \in (\mathbb N\times \mathbb Z)^{<\omega}$.  Let $$T=\{\alpha\in (\mathbb N\times \mathbb Z)^{<\omega}:X_\alpha\neq\varnothing\text{ and }N(\alpha(i))\geq i\text{ for each }i<\dom(\alpha)\}.$$

\begin{ucl}\label{t7}$(X_\alpha)_{\alpha\in T}$ is a Sierpi\'{n}ski stratification of $X$.\end{ucl}

\begin{proof}We will verify  (1) through (4) from the definition in Section 3.1.  

Clearly $T$ is a tree over $\mathbb N\times \mathbb Z$,  and each  $X_\alpha$ is a closed subset of $X_{\varnothing}=X$. Thus (1) and (2) hold.  

To see that $X_\alpha\subset \bigcup \{X_{\beta}:\beta\in \suc(\alpha)\}$, let ${\uline{s}}\in X_\alpha$.  Since  ${\uline{s}}\in X$, there exists $N\geq\dom(\alpha)$ such that $s_n\geq \dom(\alpha)+1$ for all $n\geq N$. Let $\beta=\alpha^\frown \langle N,s_{\dom(\alpha)}\rangle$. Then ${\uline{s}}\in X_\beta$ and $\beta\in \suc(\alpha)$.  The other inclusion is trivial, so this verifies property (3). 
 
 Finally,  let   $\lambda=\langle\langle N_0,s_0\rangle, \langle N_1,s_1\rangle,\ldots\rangle\in [T]$ be given.  
The sequence $X_{\lambda\restriction 0}, X_{\lambda\restriction 1}, \ldots$ clearly converges to $\uline s$. We will prove $\uline s\in X$ by showing that $s_n\geq k+1$ for each $k<\omega$ and $n\geq N_k$. To that end, let  $k<\omega$ and suppose $n\geq N_k$.    Since $\lambda\restriction n+1\in T$,   there exists $\uline{s}^0\in X_{\lambda\restriction n+1}$.  Then $ s^0_n=s_n$. Further,  $\lambda\restriction k+1\in T$  implies $N_k\geq k$ which gives us $n\geq k$. Thus $X_{\lambda\restriction n+1}\subset X_{\lambda\restriction k+1}$ and so $\uline{s}^0\in X_{\lambda\restriction k+1}$. Therefore $s^0_n\geq k+1$. We have $s_n=s^0_n\geq k+1$ as desired. This completes the proof of (4). \end{proof}

The next three claims will establish (a), (b) and (c) of Proposition \ref{t4}.

\begin{ucl}\label{t8}$X_\alpha ':=\{\uline{s}\in X_\alpha:t_{\uline{s}}<\infty\}$ is dense in $X_\alpha$ for each $\alpha\in T$.\end{ucl}

\begin{proof}Let $\uline{s}\in X_\alpha$. For each $n\geq \dom(\alpha)$ define $\uline{s}^n\in \mathbb Z ^\omega$ by $\uline{s}^n\restriction n=\uline{s}\restriction n$ and $s^n_k=k$ for all $k\geq n$. Clearly $\uline{s}^n\to \uline{s}$ and $\uline{s}^n\in X_\alpha$.  Additionally,  $$\sup_{k\geq 1} F^{-k}(2\pi | s^n_k|)<\infty$$   because $F^{-n}(2\pi n)\geq F^{-k}(2\pi k)$ for all $k\geq n$. Hence $t_{\uline{s}^n}<\infty$ by \cite[Observation 3.7]{rem}.  Therefore  $\uline{s}^n\in X'_\alpha$. We conclude that $\uline s$ lies in the closure of $X_\alpha '$, which shows that $X_\alpha '$ is dense in $X_{\alpha}$.\end{proof}

\begin{ucl} \label{t9}$G^{\psi\restriction X_\alpha}_\infty$ is dense in $L^{\psi\restriction X'_{\alpha}}_\infty$ for each $\alpha\in T$.\end{ucl}

\begin{proof}Fix ${\uline{s}}^0\in  X'_\alpha$. We will show that $[\psi(\uline{s}^0),\infty]\times \{{\uline{s}}^0\}\subset\overline{G^{\psi\restriction X_\alpha}_\infty}$. 

Let $A_{{\uline{s}}^0}=\{{\uline{s}}\in \mathbb Z ^\omega:s_n\geq s^0_n\text{ for every }n<\omega\}\subset X$. Then 
$\{\langle t_{\uline{s}},{\uline{s}}\rangle:{\uline{s}}\in A'_{{\uline{s}}^0}\}$ is the set of endpoints of $$\bigcup _{{\uline{s}}\in A'_{{\uline{s}}^0}}[t_{\uline{s}},\infty]\times \{{\uline{s}}\}.$$ The latter set is denoted $X_{{\uline{s}}^0}(\mathcal F)$ in \cite{rem}. By  \cite[Theorem 3.6]{rem}, the endpoints of  $X_{{\uline{s}}^0}(\mathcal F)$ are dense in $X_{{\uline{s}}^0}(\mathcal F)$.   Thus we have $$\overline{G^{\psi\restriction A_{{\uline{s}}^0}}_\infty}=\overline{ \{\langle t_{\uline{s}},{\uline{s}}\rangle:{\uline{s}}\in A'_{{\uline{s}}^0}\}}=\bigcup _{{\uline{s}}\in A'_{{\uline{s}}^0}}[t_{\uline{s}},\infty]\times \{{\uline{s}}\}=L^{\psi\restriction A'_{{\uline{s}}^0}}_\infty.$$ 
Now let $B_{{\uline{s}}^0}=\{{\uline{s}}\in \mathbb Z^\omega: \uline{s}\restriction \dom(\alpha)=\uline{s}^0\restriction \dom(\alpha)\}$, and let $C_{{\uline{s}}^0}=A_{{\uline{s}}^0}\cap B_{{\uline{s}}^0}\subset X_\alpha$. Since $B_{{\uline{s}}^0}$ is clopen in $\mathbb Z^\omega$, the equation above implies that $$\overline{G^{\psi\restriction C_{{\uline{s}}^0}}_\infty}=L^{\psi\restriction C'_{{\uline{s}}^0}}_\infty.$$ Therefore $[\psi(\uline{s}^0),\infty]\times \{{\uline{s}}^0\}\subset \overline{G^{\psi\restriction C_{{\uline{s}}^0}}_\infty}\subset \overline{G^{\psi\restriction X_\alpha}_\infty}$. 
\end{proof}


\begin{ucl}\label{t10} $G^{\psi\restriction X_{\beta}}_\infty$ is nowhere dense in $G^{\psi\restriction X_{\alpha}}_\infty$  for each $\alpha\in T$ and $\beta\in \suc(\alpha)$.\end{ucl}

\begin{proof}Since  $G^{\psi\restriction X_{\beta}}_\infty$ is a closed subset of $G^{\psi\restriction X_\alpha}_\infty$,  it suffices to show that  $G^{\psi\restriction X_\alpha\setminus X_{\beta}}_\infty$ is dense in $G^{\psi\restriction X_{\alpha}}_\infty$. To that end, let $\langle t_{\uline{s}},\uline{s}\rangle \in G^{\psi\restriction X_\alpha}_\infty$.  

We will demonstrate that a sequence of points in $G_\infty^{\psi\restriction X_\alpha\setminus X_{\beta}}$ converges to $\langle t_{\uline{s}},\uline{s}\rangle$.  For each $n<\omega$ define $\uline{s}^n\in \mathbb Z ^\omega$ by ${s}^n_i={s}_i$ for all $i\neq n$, and ${s}^n_n=\min\{s_n,\dom(\alpha)\}.$   Clearly $|s^n_i|\leq |s_i|$ for all $i<\omega$. So  $ t_{\uline{s}^n}\leq t_{\uline{s}}$  by \cite[Observation 3.7]{rem}. Then $\langle t_{\uline{s}^n},\uline{s}^n\rangle\to \langle t_{\uline{s}},\uline{s}\rangle$ by lower semi-continuity of $\psi$.    Note that  $$\langle t_{\uline{s}^n},\uline{s}^n\rangle\in G^{\psi\restriction X_\alpha\setminus X_{\beta}}_\infty$$ when $n\geq N_{\dom(\alpha)}:=N(\beta(\dom(\alpha)))$. Therefore $\langle t_{\uline{s}},\uline{s}\rangle\in \overline{G_\infty ^{\psi\restriction X_\alpha\setminus X_{\beta}}}.$
\end{proof}

From Claims \ref{t7} through \ref{t10} and Proposition \ref{t4}, we conclude that $ \ddot E(f)\simeq \mathfrak E.$ 

\section{The point at infinity} In proving that $\mathfrak E$ is $1$-dimensional, Erd\H{o}s showed that $\mathfrak E\cup \{\infty\}$ is connected, where $\infty$ is the single point needed to compactify $\ell^2$. 

We have shown that $\ddot E(f)$ is a topological embedding of $\mathfrak E$ into the complex plane, and we can add that $\ddot E(f) \cup\{\infty\}$ is connected, $\infty$ here being the point at infinity on the Riemann sphere. This we will argue  using the homeomorphism $H$ from Proposition \ref{ppl} and the sets defined  in Claim 4.3. 

\begin{ut}$\ddot E(f) \cup\{\infty\}$ is connected. \end{ut}

\begin{proof}For every $\uline s^0\in X'$,$$H\big(L^{\psi\restriction A'_{{\uline{s}}^0}}_\infty\big)\cup \{\infty\}$$ is a Lelek fan (a smooth fan with a dense set of endpoints) by  \cite[Theorem 3.6]{rem}. The endpoint set of any Lelek fan becomes connected when the ramification point is added to it  \cite[Section 2]{rem}. Thus $$H\big(G^{\psi\restriction A_{{\uline{s}}^0}}_\infty\big)\cup \{\infty\}$$ is a connected subset of $ \ddot E(f)\cup \{\infty\}$ for every $\uline s^0\in X'$. Every point of $\ddot E(f)$ is contained in a set of that form. Therefore $ \ddot E(f)\cup \{\infty\}$  can be written as a union of connected sets each containing the point $\infty$.\end{proof}

This makes $\ddot E(f)$ a particularly nice embedding of $\mathfrak E$. 
For complete Erd\H{o}s space $\mathfrak E_{\mathrm{c}}$ this type of embedding was discovered  by Kawamura, Oversteegen, and Tymchatyn when they proved $E(f)\simeq \mathfrak E_{\mathrm{c}}$ \cite{31}. Mayer \cite{may} had already established the connectedness of $E(f)\cup \{\infty\}$.


\begin{thebibliography}{HD}

\bibitem{aa}J. M. Aarts, L. G. Oversteegen, The geometry of Julia sets, Trans. Amer. Math. Soc.
338 (1993), no. 2, 897--918.

\bibitem{rem}N. Alhabib, L. Rempe-Gillen, Escaping Endpoints Explode. Comput. Methods Funct.
Theory 17, 1 (2017), 65--100.


\bibitem{dev}R. L. Devaney and M. I. Krych, Dynamics of $\exp(z)$, Ergodic Theory Dynam. Systems 4  (1984), no. 1, 35--52.

\bibitem{bif}R. L. Devaney, $e^z$ dynamics and bifurcations, Internat. J. Bifurcations and Chaos 1 (1991), 287--308.

\bibitem{erd}J. J. Dijkstra, J. van Mill, Erd\H{o}s space and homeomorphism groups of manifolds, Mem. Amer. Math. Soc. 208 (2010), no. 979.

\bibitem{erem}A. E. Eremenko, On the iteration of entire functions. - In: Dynamical systems and ergodic theory (Warsaw, 1986), Banach Center Publ. 23, PWN, Warsaw, 1989, 339--345.



\bibitem{dims}P. Erd\H{o}s, The dimension of the rational points in Hilbert space, Ann. of Math. (2) 41 (1940), 734--736.




\bibitem{31}K. Kawamura, L. G. Oversteegen, and E. D. Tymchatyn, On homogeneous totally disconnected $1$-dimensional spaces, Fund. Math. 150 (1996), 97--112.

\bibitem{lip}D. S. Lipham, A note on the topology of escaping endpoints. Ergodic Theory Dynam. Systems 41 (2021), no. 4, 1156--1159.

\bibitem{lipp} D. S. Lipham, Distinguishing endpoint sets from Erdős space. Math. Proc. Cambridge Philos. Soc. 173 (2022), no. 3, 635--646.


\bibitem{may} J. C. Mayer, An explosion point for the set of endpoints of the Julia set of $\exp(z)$, Ergodic
Theory Dynam. Systems 10 (1990), 177--183.


\bibitem{rem2}L. Rempe, Topological dynamics of exponential maps on their escaping sets. Ergodic Theory Dynam. Systems 26(6) (2006),  1939--1975.

\bibitem{sz}D. Schleicher, J. Zimmer, Escaping points of exponential maps, J. Lond. Math. Soc. 67(2) (2003),  380--400.

\end{thebibliography}
\end{document}